\documentclass[11pt]{amsart}
\usepackage{amssymb}
\usepackage{epsfig}
\usepackage{mathdots}
 \theoremstyle{plain}
\newtheorem{theorem}{Theorem}
\newtheorem{corollary}{Corollary}

\newtheorem{lemma}{Lemma}
\newtheorem{proposition}{Proposition}
\newtheorem{example}{Example}
\theoremstyle{definition}
\newtheorem{definition}{Definition}
\theoremstyle{remark}

\numberwithin{equation}{section}

\newcommand{\bT}{\begin{theorem}}
\newcommand{\eT}{\end{theorem}}
\newcommand{\bProp}{\begin{proposition}}
\newcommand{\eProp}{\end{proposition}}
\newcommand{\bE}{\begin{example}}
\newcommand{\eE}{\end{example}}
\newcommand{\bL}{\begin{lemma}}
\newcommand{\eL}{\end{lemma}}
\newcommand{\bP}{\begin{proof}}
\newcommand{\eP}{\end{proof}}
\newcommand{\bC}{\begin{corollary}}
\newcommand{\eC}{\end{corollary}}
\newcommand{\bD}{\begin{definition}}
\newcommand{\eD}{\end{definition}}
\newcommand{\be}{\begin{enumerate}}
\newcommand{\ee}{\end{enumerate}}
\newcommand{\beqa}{\begin{eqnarray*}}
\newcommand{\eeqa}{\end{eqnarray*}}
\newcommand{\beqaa}{\begin{eqnarray}}
\newcommand{\eeqaa}{\end{eqnarray}}
\newcommand{\ba}{\begin{array}}
\newcommand{\ea}{\end{array}}

\setbox0=\hbox{$+$}
\newdimen\plusheight
\plusheight=\ht0
\def\+{\;\lower\plusheight\hbox{$+$}\;}

\setbox0=\hbox{$-$}
\newdimen\minusheight
\minusheight=\ht0
\def\-{\;\lower\minusheight\hbox{$-$}\;}

\setbox0=\hbox{$\cdots$}
\newdimen\cdotsheight
\cdotsheight=\plusheight
\def\cds{\lower\cdotsheight\hbox{$\cdots$}}
\begin{document}

\title[General WP-Bailey Chains ]
       {General WP-Bailey Chains}

\author{James Mc Laughlin}
\address{Mathematics Department\\
 Anderson Hall\\
West Chester University, West Chester, PA 19383}
\email{jmclaughl@wcupa.edu}

\author{Peter Zimmer}
\address{Mathematics Department\\
 Anderson Hall\\
West Chester University, West Chester, PA 19383}
\email{pzimmer@wcupa.edu}

 \keywords{
 Bailey chains, WP-Bailey Chains, WP-Bailey pairs}
 \subjclass[2000]{Primary: 33D15. Secondary:11B65, 05A19.}

\date{\today}

\begin{abstract}
Motivated by a recent paper of Liu and Ma, we describe a number of
general WP-Bailey chains. We show that many of the existing
WP-Bailey chains (or branches of the WP-Bailey tree), including
chains found by Andrews, Warnaar and Liu and Ma, arise as special
cases of these general WP-Bailey chains.

We exhibit three new branches of the WP-Bailey tree, branches which
also follow as special cases of these general WP-Bailey chains.

Finally, we describe a number of new transformation formulae for
basic hypergeometric series which arise as consequences of these new
WP-Bailey chains.
\end{abstract}

\maketitle

\section{Introduction}

Andrews \cite{A01}, following on from  prior work of Bressoud
\cite{B81a}
 and Singh \cite{S94}, defined a \emph{WP-Bailey
pair} to be a pair of sequences $(\alpha_{n}(a,k,q)$,
$\beta_{n}(a,k,q))$ (if the context is clear, we occasionally
suppress the dependence on $q$ and write $(\alpha_{n}(a,k),
\beta_{n}(a,k))$) satisfying {\allowdisplaybreaks
\begin{align}\label{WPpair}
\beta_{n}(a,k,q) &= \sum_{j=0}^{n}
\frac{(k/a;q)_{n-j}(k;q)_{n+j}}{(q;q)_{n-j}(aq;q)_{n+j}}\alpha_{j}(a,k,q)\\
&= \frac{(k/a,k;q)_n}{(aq,q;q)_n}\sum_{j=0}^{n}
\frac{(q^{-n};q)_{j}(kq^n;q)_{j}}{(aq^{1-n}/k;q)_{j}(aq^{n+1};q)_{j}}
\left(\frac{qa}{k}\right)^j\alpha_{j}(a,k,q). \notag
\end{align}
}

Andrews also showed in \cite{A01} that there were two distinct ways
to construct new WP-Bailey pairs from a given pair (see \eqref{wpn1}
and \eqref{wpn2} below). These two constructions allowed a ``tree"
of WP-Bailey pairs to be generated from a single WP-Bailey pair.
Andrews and Berkovich \cite{AB02} further investigated these two
branches of the WP-Bailey tree, in the process deriving many new
transformations for basic hypergeometric series. Spiridonov
\cite{S02} derived an elliptic generalization of Andrews first
WP-Bailey chain, and Warnaar \cite{W03}  added four new branches to
the WP-Bailey tree (see \eqref{wpn3}, \eqref{wpn6}, \eqref{wpn6a}
and \eqref{wpn7} below), two of which had generalizations to the
elliptic level. More recently, and motivated in part by the papers
above, Liu and Ma \cite{LM08} introduced the idea of a general
WP-Bailey chain (as a solution to a system of linear equations), and
added one new branch to the WP-Bailey tree (see \eqref{wpn4} below).

Motivated by Liu and Ma's concept of a general WP-Bailey pair, in
the present paper we reformulate their construction in terms of the
more standard $(a,k)$ notation and also derive three other general
WP-Bailey chains. We show that all the WP-Bailey chains referred to
above arise as special cases of these general chains.

We further derive three new WP-Bailey chains from these general
chains. Lastly, we find many new transformations of basic
hypergeometric series, by inserting known WP-Bailey pairs into these
three new WP-Bailey chains.

\section{The General WP-Bailey Chain of Liu and Ma}

Liu and Ma \cite{LM08} defined a WP-Bailey chain as follows:
\begin{definition}
\emph{ Any pair of sequences $\{\alpha_{n}(t,b)\}$ and
$\{\beta_n(t,b)\}$ satisfying
\begin{equation}\label{wpceq}
\beta_n(t,b)=\sum_{i=0}^n
\frac{(bt;q)_{n+i}(b;q)_{n-i}}{(q;q)_{n-i}(tq;q)_{n+i}}\alpha_i(t,b)
\end{equation}
is called a well-poised (in short, WP) Bailey pair w.r.t. the
parameters $t$ and $b$, denoted by $(\alpha_{n}(t,b),
\beta_n(t,b))$. Further, if for $n \geq 0$, the sequences
$\{\alpha_{n}'(t,c)\}$ and $\{\beta_n'(t,c)\}$ defined by
\begin{align}\label{newabs}
\alpha_{n}'(t,c)&=\lambda_n\alpha_{n}(t,b),\\
\beta_n'(t,c) &=\sum_{i=0}^n d_{n,i}\beta_{i}(t,b) \notag
\end{align}
is still a WP-Bailey pair w.r.t. the parameters $t$ and $c$, then
the iterative process from $(\alpha_{n}(t,b), \beta_n(t,b))$ to
$(\alpha_{n}'(t,c), \beta_n'(t,c))$ is said to be a WP-Bailey
chain.}
\end{definition}

The authors point out in \cite{LM08} that the $\lambda_n$ and the
$d_{n,i}$ always satisfy
\begin{equation}\label{dinlneq}
\sum_{i=k}^n
\frac{(bt;q)_{i+k}(b;q)_{i-k}}{(q;q)_{i-k}(tq;q)_{i+k}}d_{n,i} =
\frac{(ct;q)_{n+k}(c;q)_{n-k}}{(q;q)_{n-k}(tq;q)_{n+k}}\lambda_k,
\end{equation}
and then proceed, using infinite matrices, to invert \eqref{dinlneq}
to find expressions for the $d_{n,k}$ as sums over the $\lambda_i$.

For our purposes, we reformulate their result using the more usual
notation in which $(\alpha_{n}(a,k)$, $\beta_{n}(a,k))$ denotes a
WP-Bailey pair, using a single sequence $g_n$ (replacing their
$\lambda_n$), and giving a more direct proof which avoids the use of
infinite matrices. Before proceeding to this, we first recall the
following necessary result of Warnaar.
\begin{lemma}[Warnaar, \cite{W03}]\label{Wlem}
For $a$ and $k$ indeterminates the following equations are
equivalent:
\begin{align}\label{WPequiv}
\beta_{n}(a,k) &= \sum_{j=0}^{n}
\frac{(k/a)_{n-j}(k)_{n+j}}{(q)_{n-j}(aq)_{n+j}}\alpha_{j}(a,k),\\
\alpha_{n}(a,k) &= \frac{1-a q^{2n}}{{1-a}}\sum_{j=0}^{n}\frac{1-k
q^{2j}}{{1-k}}
\frac{(a/k)_{n-j}(a)_{n+j}}{(q)_{n-j}(kq)_{n+j}}\left(\frac{k}{a}
\right)^{n-j}\beta_{j}(a,k). \notag
\end{align}
\end{lemma}
If $a$ and $k$ are interchanged in the second equation, we also note
for later that Lemma \ref{Wlem} implies the following.
\begin{corollary}\label{warcor}
If $(\alpha_{n}(a,k), \beta_n(a,k))$ are a WP-Bailey pair, then so
are\\ $(\alpha_{n}'(a,k), \beta_n'(a,k))$, where
\begin{align*}
\alpha_{n}'(a,k)&=\frac{1-a q^{2n}}{1-a}\left (\frac{k}{a} \right)^n
\beta_{n}(k,a),\\
\beta_n'(a,k)&= \frac{1-k}{1-k q^{2n}}\left (\frac{k}{a} \right)^n
\alpha_{n}(k,a).
\end{align*}
\end{corollary}

\begin{theorem}\label{gwct}
Suppose that $(\alpha_{n}(a,k),\,\beta_{n}(a,k))$ satisfy
\eqref{WPpair}, that
 $g_n$ is an arbitrary sequence of functions and $c$ and $e$ are
arbitrary constants. Then $(\alpha_{n}'(a,k)$, $\beta_{n}'(a,k))$
also satisfy \eqref{WPpair}, where
\begin{align}\label{gwp}
\alpha_n'&(a,k) = g_n\alpha_n(e,c),\\
\beta_{n}'&(a,k)= \sum_{j=0}^{n}\beta_{j}(e,c)\frac{(1-c
q^{2j})(k/a;q)_{n-j}(k;q)_{n+j}(eq;q)_{2j}}{(1-c
)(q;q)_{n-j}(aq;q)_{n+j}(cq;q)_{2j}} \notag\\
&\times  \sum_{r=0}^{n-j}
\frac{(1-eq^{2j+2r})(q^{-(n-j)},kq^{n+j},e/c,eq^{2j};q)_r}{(1-e
q^{2j})(a q^{n+j+1},a q^{1-(n-j)}/k,cq^{2j+1},q;q)_r}\left(\frac{q a
c}{k e} \right)^r g_{r+j}. \notag
\end{align}
\end{theorem}

\begin{proof}
If $\alpha_n'(a,k)$ is as given, then {\allowdisplaybreaks
\begin{align*}
\beta_{n}'(a,k)&=\sum_{r=0}^{n}
\frac{(k/a)_{n-r}(k)_{n+r}}{(q)_{n-r}(aq)_{n+r}}g_r\alpha_{r}(e,c)\\
&=\sum_{r=0}^{n}\frac{(k/a)_{n-r}(k)_{n+r}}{(q)_{n-r}(aq)_{n+r}}g_r\\
&\phantom{asdasdas}\times
\frac{1-eq^{2r}}{{1-e}}\sum_{j=0}^{r}\frac{1-c
q^{2j}}{{1-c}}\frac{(e/c)_{r-j}(e)_{r+j}}{(q)_{r-j}(cq)_{r+j}}\left(\frac{c}{e}
\right)^{r-j}\beta_{j}(e,c)\\
&=\sum_{j=0}^{n}\frac{1-c q^{2j}}{{1-c}}\beta_{j}(e,c)\\
&\phantom{asadasdadf}\times\sum_{r=j}^{n}\frac{(k/a)_{n-r}(k)_{n+r}}
{(q)_{n-r}(aq)_{n+r}}g_r\frac{1-eq^{2r}}{{1-e}}
\frac{(e/c)_{r-j}(e)_{r+j}}{(q)_{r-j}(cq)_{r+j}}\left(\frac{c}{e}
\right)^{r-j}\\
&=\sum_{j=0}^{n}\frac{1-c q^{2j}}{{1-c}}\beta_{j}(e,c)\\
&\phantom{asas}\times\sum_{r=0}^{n-j}\frac{(k/a)_{n-r-j}(k)_{n+r+j}}
{(q)_{n-r-j}(aq)_{n+r+j}} g_{r+j}\frac{1-eq^{2r+2j}}{{1-e}}
\frac{(e/c)_{r}(e)_{r+2j}}{(q)_{r}(cq)_{r+2j}}\left(\frac{c}{e}
\right)^{r}\\
&=\sum_{j=0}^{n}\frac{1-c q^{2j}}{{1-c}}\beta_{j}(e,c)
\frac{(k/a)_{n-j}(k)_{n+j}(eq)_{2j}}{(q)_{n-j}(aq)_{n+j}(cq)_{2j}}\\
&\phantom{}\times\sum_{r=0}^{n-j}\frac{1-eq^{2r+2j}}{{1-e q^{2j}}}
\frac{(q^{-(n-j)},k q^{n+j},e/c,e q^{2j};q)_r}{(a q^{1-(n-j)}/k,a
q^{n+j+1},c q^{2j+1},q;q)_r} \left(\frac{q a c}{k e}
\right)^{r}g_{r+j}.
\end{align*}
}
\end{proof}

Of course, from the perspective of adding branches to the WP-Bailey
tree, what is desirable is to choose $e$ and $c$ and the sequence
$g_n$ so as to make the inner sum over $r$ above have closed form.

It is shown in \cite{LM08} that Andrews' first WP-Bailey chain in
\cite{A01} follows as a special case of their theorem, and the
authors also find a new WP-Bailey chain as another special case. In
fact Andrews' second WP-Bailey chain in \cite{A01} and the first
WP-Bailey chain of Warnaar in \cite{W03} also follow as special
cases. Since our notation is somewhat different from that of Liu and
Ma \cite{LM08}, for completeness we show how all four WP Bailey
chains follow as special cases of Theorem \ref{gwct}.

 \begin{corollary} If $(\alpha_{n}(a,k),\,\beta_{n}(a,k))$ satisfy
\eqref{WPpair}, then so do $(\alpha_{n}'(a,k)$,\\
$\beta_{n}'(a,k))$
 and $(\tilde{\alpha}_{n}(a,k),\,\tilde{\beta}_{n}(a,k))$ \cite{A01},
 $(\alpha_{n}^{\dagger}(a,k),\,\beta_{n}^{\dagger}(a,k))$ \cite{W03} and
 $(\alpha_{n}^{*}(a,k)$, $\beta_{n}^{*}(a,k))$ \cite{LM08}, where
{\allowdisplaybreaks
\begin{align}\label{wpn1}
\alpha_{n}'(a,k)&=\frac{(\rho_1, \rho_2)_n}{(aq/\rho_1,
aq/\rho_2)_n}\left(\frac{aq}{\rho_1 \rho_2}\right)^n\alpha_{n}(a,c),
\hspace{25pt} \text{$($Andrews \cite{A01}$)$}\\
\beta_{n}'(a,k)&=\frac{(k\rho_1/a,k\rho_2/a)_n}{(aq/\rho_1,
aq/\rho_2)_n}\times  \notag\\
&\sum_{j=0}^{n} \frac{(1-c
q^{2j})(\rho_1,\rho_2)_j(k/c)_{n-j}(k)_{n+j}}{(1-c)(k\rho_1/a,k\rho_2/a)_n(q)_{n-j}(qc)_{n+j}}
\left(\frac{aq}{\rho_1 \rho_2}\right)^j\beta_{j}(a,c),\notag
\end{align}
} with $c=k\rho_1 \rho_2/aq$;
\begin{align}\label{wpn2}
\tilde{\alpha}_{n}(a,k)&= \frac{(qa^2/k)_{2n}}{(k)_{2n}}\left
(\frac{k^2}{q a^2} \right)^n\alpha_{n} \left(a, c
\right), \hspace{25pt} \text{$($Andrews \cite{A01}$)$}\\
\tilde{\beta}_{n}(a,k)&=\sum_{j=0}^{n}
\frac{(k^2/qa^2)_{n-j}}{(q)_{n-j}}\left (\frac{k^2}{q a^2}
\right)^j\beta_{j} \left(a, c \right), \notag
\end{align}
with  $c=q a^2/k$;
{\allowdisplaybreaks
\begin{align}\label{wpn3}
& \text{$($Warnaar \cite{W03}$)$}\\
\alpha_{n}^{\dagger}(a,k)&=\frac{1-k^{1/2}}{1-q^n
k^{1/2}}\frac{1+q^n a/k^{1/2}} {1+a/k^{1/2}}\frac{(a^2/k;q)_{2n}}
{(k;q)_{2n}}\left(\frac{k^2}{a^2}\right)^n\alpha_{n}(a,c),\notag\\
\beta_{n}^{\dagger}(a,k)&=\frac{1-k^{1/2}}{1-q^n k^{1/2}}
\sum_{j=0}^{n} \frac{1+q^j a/k^{1/2}}
{1+a/k^{1/2}}\frac{(k/c;q)_{n-j}}{(q;q)_{n-j}}\left(\frac{k^2}{a^2}\right)^j\beta_{j}(a,c),
\notag
\end{align}
} with  $c=a^2/k$ (we consider just the $\sigma =  1$ case of
Warnaar's theorem here);
 {\allowdisplaybreaks
\begin{align}\label{wpn4}
& \text{$($Liu and Ma \cite{LM08}$)$}\\
\alpha_{n}^{*}(a,k)&=\frac{(a^2/k;q^2)_{n}}
{(kq;q^2)_{n}}\left(\frac{-k}{a}\right)^n\alpha_{n}(a,c),\notag\\
\beta_{n}^{*}(a,k)&=\sum_{j=0}^{\lfloor n/2 \rfloor}
\frac{1-\frac{a^2q^{2n-4j}}{k}}
{1-\frac{a^2}{k}}\frac{(k;q^2)_{n-j}(k^2/a^2;q^2)_j}{(a^2
q^2/k;q^2)_{n-j}(q^2;q^2)j}\left(\frac{-k}{a}\right)^{n-2j}\beta_{n-2j}(a,c),
\notag
\end{align}
}with  $c=a^2/k$.
\end{corollary}
\begin{proof}
In each case we apply Theorem \ref{gwct}, letting $e=a$ and $g_n$ be
the factor multiplying $\alpha_{n}(a,c)$ in each of the WP-Bailey
chains listed above. That the inner sum over $r$ in \eqref{gwp} has
a closed form and that the stated expressions for the various
$\beta_n^{\text{new}}(a,k)$ above now hold, follow as a consequence
of, respectively, Jackson's sum of a terminating $_8 \phi_7$,
\begin{multline}\label{jackson}
_{8} \phi _{7} \left [
\begin{matrix}
a, \,q a^{1/2},\, -q a^{1/2}, \,b,\, c,\,d,\,e, \, q^{-n}\\
a^{1/2},\,-a^{1/2},\, aq/b,\,aq/c,\,aq/d,\,aq/e,\, a q^{n+1}
\end{matrix}
; q,q \right ]\\ = \frac{(a
q,aq/bc,aq/bd,aq/cd;q)_n}{(aq/b,aq/c,aq/d,aq/bcd;q)_n},
\end{multline}
where $a^2q=bcdeq^{-n}$,  the following formulae (see \cite[problem
2.12]{GR04} and \cite[Equation (3.2)]{AB02}, respectively):
\begin{multline}\label{GR}
_{10} \phi _{9} \left [
\begin{matrix}
a, \,q \sqrt{a},\, -q \sqrt{a}, \,a \sqrt{\frac{q}{k}},\, -a
\sqrt{\frac{q}{k}},
\,\frac{a q}{ \sqrt{k}},-\frac{a q}{ \sqrt{k}},\,\frac{k}{a q},\,k q^n, \, q^{-n}\\
\sqrt{a},\,-\sqrt{a},\, \sqrt{k q},\,-\sqrt{k q},
\,\sqrt{k},\,-\sqrt{k},\,\frac{a^2 q^2}{k}, \frac{a q^{1-n}}{k},\, a
q^{n+1}
\end{matrix}
; q,q \right ]\\
= \frac{\left(a q,\frac{k^2}{q
a^2};q\right)_n}{\left(k,\frac{k}{a};q\right)_n},
\end{multline}
{\allowdisplaybreaks
\begin{multline}\label{AB}
_{10} \phi _{9} \left [
\begin{matrix}
a, \,q \sqrt{a},\, -q \sqrt{a}, \,a \sqrt{\frac{q}{k}},\, -a
\sqrt{\frac{q}{k}},
\,\frac{a}{ \sqrt{k}},-\frac{a q}{ \sqrt{k}},\,\frac{k}{a},\,k q^n, \, q^{-n}\\
\sqrt{a},\,-\sqrt{a},\, \sqrt{k q},\,-\sqrt{k q}, \,q
\sqrt{k},\,-\sqrt{k},\,\frac{a^2 q}{k}, \frac{a q^{1-n}}{k},\, a
q^{n+1}
\end{matrix}
; q,q \right ]\\
= \frac{\left(a
q,\sqrt{k},\frac{k^2}{a^2};q\right)_n}{\left(k,\frac{k}{a},q
\sqrt{k};q\right)_n},
\end{multline}
} and a $q$-analogue of Watson's $_3 F_2$ sum (see
\cite[II.16]{GR04}),
\begin{multline*}
_{8} \phi _{7} \left [
\begin{matrix}
\lambda, \,q \sqrt{\lambda},\, -q \sqrt{\lambda}, \,q^{-n},\,
bq^{n},\, \lambda \sqrt{q/b},\,-\lambda \sqrt{q/b}, \,
b/\lambda\\
\sqrt{\lambda},\,-\sqrt{\lambda},\, \lambda q^{n+1},\,\lambda
q^{1-n}/b,\,\lambda^2 q/b,\,\sqrt{q b},\, -\sqrt{q b}
\end{matrix}
; q,-\frac{q \lambda}{b}  \right ]\\ = \begin{cases} 0,& n \text{
odd} \\
\displaystyle{\frac{(\lambda
q;q)_{n}}{(b/\lambda;q)_{n}}\frac{(q,b^2/ \lambda^2;q^2)_{n/2}}{( b
q, q^2\lambda^2/b;q^2)_{n/2}}},& n \text{ even}.
\end{cases}
\end{multline*}
\end{proof}

Before describing and proving a new WP-Bailey chain which follows
from Theorem \ref{gwct}, we need a preliminary result.
\begin{lemma}
\begin{multline}\label{M1}
_{10} \phi _{9} \left [
\begin{matrix}
a, \,q \sqrt{a},\, -q \sqrt{a}, \,a \sqrt{\frac{q}{k}},\, -a
\sqrt{\frac{q}{k}},
\,\frac{a}{ \sqrt{k}},-\frac{a}{ \sqrt{k}},\,\frac{k q}{a},\,
k q^n, \, q^{-n}\\
\sqrt{a},\,-\sqrt{a},\, \sqrt{k q},\,-\sqrt{k q},
\,q\sqrt{k},\,-q\sqrt{k},\,\frac{a^2}{k}, \frac{a q^{1-n}}{k},\, a
q^{n+1}
\end{matrix}
; q,q \right ]\\
=\frac{1-k}{1-k q^{2n}} \frac{\left(a q,\frac{k^2q}{
a^2};q\right)_n}{\left(k,\frac{k}{a};q\right)_n}.
\end{multline}
\end{lemma}

\begin{proof}
In Bailey's transformation (see \cite{GR04}, III.28),
\begin{multline}\label{B1}
_{10} \phi _{9} \left [
\begin{matrix}
a, \,q \sqrt{a},\, -q \sqrt{a}, \,b,\, c, \,d,\,e,\,f,\,
\lambda a q^{n+1}/ef, \, q^{-n}\\
\sqrt{a},\,-\sqrt{a},\, aq/b,\,aq/c, \,aq/d,\,aq/e,\,aq/f, e f
q^{-n}/\lambda,\, a q^{n+1}
\end{matrix}
; q,q \right ]\\
=\frac{\left(a q,aq/ef,\lambda q/e,\lambda
q/f;q\right)_n}{\left(aq/e,aq/f,\lambda q/ef,\lambda
q;q\right)_n}\\
\times \, _{10} \phi _{9} \left [
\begin{matrix}
\lambda, \,q \sqrt{\lambda},\, -q \sqrt{\lambda}, \lambda
b/a,\,\lambda c/a, \,\lambda d/a,\,e,\,f,\,
\lambda a q^{n+1}/ef, \, q^{-n}\\
\sqrt{\lambda},\,-\sqrt{\lambda},\,  aq/b,\,aq/c, \,aq/d,\,\lambda
q/e,\,\lambda q/f, e f q^{-n}/a,\, \lambda q^{n+1}
\end{matrix}
; q,q \right ],
\end{multline}
where $\lambda = q a^2/bcd$, set $b=a\sqrt{q/k}$, $c=-a\sqrt{q/k}$,
$d=k/aq$, $e=aq/\sqrt{k}$ and $f=-aq/\sqrt{k}$, so that $\lambda =
-aq$. The left side of \eqref{B1} becomes the left side of
\eqref{GR}, and so equals the right side of \eqref{GR}. The result
follows upon replacing $a$ with $-a/q$, after some simple
manipulations.
\end{proof}

The next WP-Bailey chain appears to be new.

\begin{corollary}\label{cn11} If $(\alpha_{n}(a,k,q),\,\beta_{n}(a,k,q))$ satisfy
\eqref{WPpair}, then so do $(\alpha_{n}'(a,k,q)$,\\
$\beta_{n}'(a,k,q))$, where {\allowdisplaybreaks
\begin{align}\label{wpn11}
\alpha_{n}'(a,k,q)&=\frac{\left(\frac{a^2}{k};q\right)_{2n}}{(kq;q)_{2n}}
\left(\frac{k^2q}{a^2} \right)^n
\alpha_{n}\left(a,\frac{a^2}{kq},q\right),\\
\beta_{n}'(a,k,q)&=\frac{1-k}{1-k q^{2n}}\sum_{j=0}^{n}
\frac{\left(1-\frac{a^2 q^{2j}}{kq}\right)}{\left(1-\frac{a^2}{kq}
\right)} \frac{\left( \frac{k^2 q}{a^2};q\right)_{n-j}}{(q;q)_{n-j}}
\left(\frac{k^2q}{a^2}
\right)^j\beta_{j}\left(a,\frac{a^2}{kq},q\right).\notag
\end{align}
}
\end{corollary}
\begin{proof}
Set $e=a$, $c=a^2/kq$, $g_n=(k^2q/a^2)^n(a^2/k;q)_{2n}/(kq;q)_{2n}$
in Theorem \ref{gwct}, and use \eqref{M1}, with $a$ replaced with $a
q^{2j}$ and $k$ with $k q^{2j}$, to put the inner sum over $r$ at
\eqref{gwp} in closed form, and simplify the resulting sum.
\end{proof}

Remarks: 1) We note the similarities between this WP-Bailey chain
and the second chain of Andrews at \eqref{wpn2}, including the fact
that a double application of this construction returns the original
WP-Bailey pair.

2) Upon inserting the unit WP-Bailey pair \eqref{up} in this chain,
we get the pair
\begin{align}\label{l1pr2}
\alpha_n(a,k)&=\frac{1-a q^{2n}}{1-a}\frac{(a,k q/a;q)_n}{(q,
a^2/k;q)_n}\frac{(a^2/k;q)_{2n}}{(k q;q)_{2n}}\left(\frac{k}{a}\right)^n,\\
\beta_n(a,q)&=\frac{1-k}{1-k q^{2n}}\frac{(q k^2/
a^2;q)_n}{(q;q)_n}, \notag
\end{align}
which can also be derived by applying Corollary \ref{warcor} to the
pair at \eqref{mz01}.

3) The small number of summation formulae for terminating basic
hypergeometric series would tend to suggest that there are not many
WP-Bailey chains which follow from Theorem \ref{gwct}. With this in
mind, we turn to  new general WP-Bailey chains.

\section{New General WP-Bailey Chains and their Consequences}
Motivated by Warnaar's other WP-Bailey chains in \cite{W03}, we now
consider some new general WP-Bailey chains.

\begin{theorem}\label{gwct2}
Suppose that $(\alpha_{n}(a,k,q),\,\beta_{n}(a,k,q))$ satisfy
\eqref{WPpair}, that
 $g_n$ is an arbitrary sequence of functions and that $e$ and $c$
 are
arbitrary constants. Then $(\alpha_{n}'(a,k,q)$,
$\beta_{n}'(a,k,q))$ also satisfy \eqref{WPpair}, where
\begin{align}\label{gwp2}
&\alpha_n'(a,k,q) = g_n\alpha_n(e,c,q^2),\\
&\beta_{n}'(a,k,q)= \sum_{j=0}^{n}\beta_{j}(e,c,q^2)\frac{(1-c
q^{4j})(k/a;q)_{n-j}(k;q)_{n+j}(eq^2;q^2)_{2j}}{(1-c
)(q;q)_{n-j}(aq;q)_{n+j}(cq^2;q^2)_{2j}}\notag\\
&\times \sum_{r=0}^{n-j}
\frac{(1-eq^{4j+4r})(q^{-(n-j)},kq^{n+j};q)_r(e/c,eq^{4j};q^2)_r
g_{r+j}}{(1-e q^{4j})(a q^{n+j+1},a
q^{1-(n-j)}/k;q)_r(cq^{4j+2},q^2;q^2)_r}\left(\frac{q c a}{k e}
\right)^r.\notag
\end{align}
\end{theorem}

\begin{proof}
As previously, if $\alpha_n'(a,k,q)$ is as given, then
{\allowdisplaybreaks
\begin{align*}
\beta_{n}'(a,k,q)&=\sum_{r=0}^{n}
\frac{(k/a)_{n-r}(k)_{n+r}}{(q)_{n-r}(aq)_{n+r}}g_r\alpha_{r}(e,c,q^2)\\
&=\sum_{r=0}^{n}\frac{(k/a)_{n-r}(k)_{n+r}}{(q)_{n-r}(aq)_{n+r}}g_r\\
&\times \frac{1-eq^{4r}}{{1-e}}\sum_{j=0}^{r}\frac{1-c
q^{4j}}{{1-c}}\frac{(e/c;q^2)_{r-j}(e;q^2)_{r+j}}{(q^2;q^2)_{r-j}
(cq^2;q^2)_{r+j}}\left(\frac{c}{e}
\right)^{r-j}\beta_{j}(e,c,q^2)\\
&=\sum_{j=0}^{n}\frac{1-c q^{4j}}{{1-c}}\beta_{j}(e,c,q^2)\\
&\times\sum_{r=j}^{n}\frac{(k/a)_{n-r}(k)_{n+r}}
{(q)_{n-r}(aq)_{n+r}}g_r\frac{1-eq^{4r}}{{1-a}}
\frac{(e/c;q^2)_{r-j}(e;q^2)_{r+j}}{(q^2;q^2)_{r-j}(cq^2;q^2)_{r+j}}
\left(\frac{c}{e}
\right)^{r-j}\\
&=\sum_{j=0}^{n}\frac{1-c
q^{4j}}{{1-c}}\beta_{j}(e,c,q^2)\sum_{r=0}^{n-j}\frac{(k/a)_{n-r-j}(k)_{n+r+j}}
{(q)_{n-r-j}(aq)_{n+r+j}}\\
& \phantom{asdasdasda}\times g_{r+j}\frac{1-eq^{4r+4j}}{{1-e}}
\frac{(e/c;q^2)_{r}(e;q^2)_{r+2j}}{(q^2;q^2)_{r}(cq^2;q^2)_{r+2j}}
\left(\frac{c}{e}
\right)^{r}\\
=&\sum_{j=0}^{n}\frac{1-c q^{4j}}{{1-c}}\beta_{j}(e,c,q^2)
\frac{(k/a)_{n-j}(k)_{n+j}(eq^2;q^2)_{2j}}{(q)_{n-j}(aq)_{n+j}(cq^2;q^2)_{2j}}
\sum_{r=0}^{n-j}\frac{1-eq^{4r+4j}}{{1-e q^{4j}}}\\&
\phantom{asdasd
}\times \frac{(q^{-(n-j)},k q^{n+j};q)_r(e/c,e
q^{4j};q^2)_r}{(a q^{1-(n-j)}/k,a q^{n+j+1};q)_r(c
q^{4j+2},q^2;q^2)_r} \left(\frac{q a c}{k e} \right)^{r}g_{r+j}.
\end{align*}
}
\end{proof}

As with Theorem \ref{gwct}, the aim is to find choices for the
sequence $g_n$ and the parameters $e$ and $c$ which lead to the
inner sum over $r$ above having closed form. We believe the
following WP-Bailey chain to be new.

 \begin{corollary}\label{cn2} If $(\alpha_{n}(a,k,q),\,\beta_{n}(a,k,q))$ satisfy
\eqref{WPpair}, then so do $(\alpha_{n}'(a,k,q)$,\\
$\beta_{n}'(a,k,q))$, where {\allowdisplaybreaks
\begin{align}\label{wpn5}
\alpha_{n}'(a,k,q)&=\frac{1+a}{1+a q^{2n}}\,q^n\,\alpha_{n}\left(a^2,\frac{ak}{q},q^2\right),\\
\beta_{n}'(a,k,q)&=\sum_{j=0}^{n}  \frac{\left(1-\frac{ak}{q}
q^{4j}\right)} {\left(1-\frac{ak}{q}\right)}
\frac{(-aq;q)_{2j}}{(-k;q)_{2j}} \frac{\left(\frac{q
k}{a};q^2\right)_{n-j}}{(q^2;q^2)_{n-j}} \frac{(k^2;q^2)_{n+j}}{(a k
q;q^2)_{n+j}}\notag \\
&\phantom{assddadsASDAsdasdaS} \times  \frac{1+a}{1+a
q^{2j}}\,q^j\beta_{j}\left(a^2,\frac{ak}{q},q^2\right).\notag
\end{align}
}
\end{corollary}
\begin{proof}
Set $e=a^2$, $c=ak/q$, $g_n=(1+a)q^n/(1+a q^{2n})$ in Theorem
\ref{gwct2}, and use \eqref{jackson} to put the inner sum over $r$
at \eqref{gwp2} in closed form, and simplify the resulting sum.
\end{proof}
Remark: Somewhat curiously, inserting the unit pair \eqref{up} in
this chain leads to a pair that is the special case $\rho_1=-\rho_2
= \sqrt{a q/k}$ of Singh's WP-Bailey pair at \eqref{singhpr}.

In Theorem \ref{gwct3} we showed how to derive  new general
WP-Bailey pairs $(\alpha_{n}'(a,k,q)$, $\beta_{n}'(a,k,q))$ in terms
of $(\alpha_{n}(e,c,q^2),\,\beta_{n}(e,c,q^2))$, where\\
$(\alpha_{n}(a,k,q),\,\beta_{n}(a,k,q))$ is an existing WP-Bailey
pair. It is also possible to derive new general WP-Bailey pairs in
terms of $(\alpha_{n}(e,c,q^{1/2}),\,\beta_{n}(e,c,q^{1/2}))$, but
for aesthetic reasons we replace $q^{1/2}$ by $q$ everywhere, to get
the following theorem.

\begin{theorem}\label{gwct3}
Suppose that $(\alpha_{n}(a,k,q),\,\beta_{n}(a,k,q))$ satisfy
\eqref{WPpair}, that
 $g_n$ is an arbitrary sequence of functions and that $e$ and $c$
 are
arbitrary constants. Then $(\alpha_{n}'(a,k,q)$,
$\beta_{n}'(a,k,q))$ also satisfy \eqref{WPpair}, where
\begin{align}\label{gwp3}
&\alpha_n'(a,k,q^2) = g_n\alpha_n(e,c,q),\\
&\beta_{n}'(a,k,q^2)= \sum_{j=0}^{n}\beta_{j}(e,c,q)\frac{(1-c
q^{2j})(k/a;q^2)_{n-j}(k;q^2)_{n+j}(eq;q)_{2j}}{(1-c
)(q^2;q^2)_{n-j}(aq^2;q^2)_{n+j}(cq;q)_{2j}}\notag\\
\times \sum_{r=0}^{n-j}&
\frac{(1-eq^{2j+2r})(q^{-2(n-j)},kq^{2n+2j};q^2)_r(e/c,eq^{2j};q)_r
g_{r+j}}{(1-e q^{2j})(a q^{2n+2j+2},a
q^{2-2(n-j)}/k;q^2)_r(cq^{2j+1},q;q)_r}\left(\frac{q^2 c a}{k e}
\right)^r.\notag
\end{align}
\end{theorem}

We omit the proof as it follows very similar lines to the proofs of
Theorems \ref{gwct} and \ref{gwct2}. Warnaar's WP-Bailey chains in
Theorems 2.4 and 2.5 of \cite{W03} follows as a special case of the
theorem above (we use $\sqrt{a}$ where Warnaar uses $a$, so as to
keep the parameters in the new WP Bailey pairs as $a$ and $k$).

\begin{corollary}[Warnaar, \cite{W03}]\label{cn3}
If $(\alpha_{n}(a,k,q),\,\beta_{n}(a,k,q))$ satisfy \eqref{WPpair},
then so do $(\alpha_{n}'(a,k,q)$,$\beta_{n}'(a,k,q))$, where
{\allowdisplaybreaks
\begin{align}\label{wpn6}
\alpha_{n}'(a,k,q^2)&=
\alpha_{n}\left(\sqrt{a},\frac{k}{q\sqrt{a}},q\right),\\
\beta_{n}'(a,k,q^2)&=\frac{\left(\frac{-k}{\sqrt{a}};q\right)_{2n}}
{(-q\sqrt{a};q)_{2n}}\sum_{j=0}^{n} \frac{\left(1-\frac{kq^{2j}
}{q\sqrt{a}}\right)}{\left(1-\frac{k}{q\sqrt{a}}\right)}
\frac{\left(\frac{aq^2}{k};q^2\right)_{n-j}}{(q^2;q^2)_{n-j}}
\frac{(k;q^2)_{n+j}}{\left( \frac{k^2
}{a};q^2\right)_{n+j}}\notag \\
&\phantom{assddadsASDAsdasdaS} \times \left(
\frac{k}{aq}\right)^{n-j}
\beta_{j}\left(\sqrt{a},\frac{k}{q\sqrt{a}},q\right).\notag
\end{align}
}
\end{corollary}
\begin{proof}
Set $e=\sqrt{a}$, $c=k/q\sqrt{a}$, $g_n=1$, use \eqref{jackson} to
get the inner sum over $r$ in \eqref{gwp3} in closed form, and the
result follows after some simple manipulations.
\end{proof}

\begin{corollary}[Warnaar, \cite{W03}]\label{cn33}
If $(\alpha_{n}(a,k,q),\,\beta_{n}(a,k,q))$ satisfy \eqref{WPpair},
then so do $(\alpha_{n}'(a,k,q)$,$\beta_{n}'(a,k,q))$, where
{\allowdisplaybreaks
\begin{align}\label{wpn6a}
\alpha&_{n}'(a,k,q^2)=
q^{-n} \frac{1+\sqrt{a}q^{2n}}{1+\sqrt{a}}\alpha_{n}\left(\sqrt{a},\frac{k}{\sqrt{a}},q\right),\\
\phantom{asasa}\beta_{n}'&(a,k,q^2)=q^{-n}\frac{\left(\frac{-kq}{\sqrt{a}};q\right)_{2n}}
{(-\sqrt{a};q)_{2n}}\sum_{j=0}^{n} \frac{\left(1-\frac{kq^{2j}
}{\sqrt{a}}\right)}{\left(1-\frac{k}{\sqrt{a}}\right)}
\frac{\left(\frac{a}{k};q^2\right)_{n-j}}{(q^2;q^2)_{n-j}}
\frac{(k;q^2)_{n+j}}{\left( \frac{k^2q^2
}{a};q^2\right)_{n+j}}\notag \\
&\phantom{assddadsASDAsdasdaS} \times \left(
\frac{k}{a}\right)^{n-j}
\beta_{j}\left(\sqrt{a},\frac{k}{\sqrt{a}},q\right).\notag
\end{align}
}
\end{corollary}
\begin{proof}
Set $e=\sqrt{a}$, $c=k/\sqrt{a}$,
$g_n=q^{-n}(1+\sqrt{a}q^n)/(1+\sqrt{a})$. To get the inner sum over
$r$ in \eqref{gwp3} in closed form, we use the following result of
Andrews and Berkovich \cite{AB02} (we replace $\sqrt{q}$ with $q$):
\begin{multline}\label{ABsum1}
_{10}\,W_9(\sqrt{a}; iqa^{1/4},-i q a^{1/4}, \sqrt{k}q^n, -
\sqrt{k}q^n, q^{-n},-q^{-n},a/k;q;q)\\
=\frac{(aq^2,a/k;q^2)_n}{(k/a,k^2/a;q^2)_n}\frac{(-k
q/\sqrt{a};q)_{2n}}{(\sqrt{a};q)_{2n}} \left( \frac{k}{a
q}\right)^n,
\end{multline}
with $a$ replaced with $aq^{4j}$, $k$ with $k q^{4j}$ and $n$ with
$n-j$. The result follows once again, after some $q$-product
manipulations.
\end{proof}

The next result also appears to be new.

\begin{corollary}\label{cn333}
If $(\alpha_{n}(a,k,q),\,\beta_{n}(a,k,q))$ satisfy \eqref{WPpair},
then so do $(\alpha_{n}'(a,k,q)$, $\beta_{n}'(a,k,q))$, where
{\allowdisplaybreaks
\begin{align}\label{wpn6aa}
\alpha&_{n}'(a,k,q^2)=
\alpha_{n}\left(\sqrt{a},\frac{k}{\sqrt{a}},q\right),\\
\phantom{asasa}\beta_{n}'&(a,k,q^2)=\frac{\left(\frac{-k}{\sqrt{a}};q\right)_{2n}}
{(-q\sqrt{a};q)_{2n}}\sum_{j=0}^{n} \frac{\left(1-\frac{k^2q^{4j}
}{a}\right)}{\left(1-\frac{k^2}{a}\right)}
\frac{\left(\frac{a}{k};q^2\right)_{n-j}}{(q^2;q^2)_{n-j}}
\frac{(k;q^2)_{n+j}}{\left( \frac{k^2q^2
}{a};q^2\right)_{n+j}}\notag \\
&\phantom{assddadsASDAsdasdaS} \times \left( \frac{k
q}{a}\right)^{n-j}
\beta_{j}\left(\sqrt{a},\frac{k}{\sqrt{a}},q\right).\notag
\end{align}
}
\end{corollary}
\begin{proof}
Set $e=\sqrt{a}$, $c=k/\sqrt{a}$, $g_n=1$. This time, to get the
inner sum over $r$ in \eqref{gwp3} in closed form, we use the
following result of Warnaar \cite{W03}:
\begin{multline}\label{Wsum1}
_{8}\,W_7(a; b, a q^n/\sqrt{b}, -
a q^n/\sqrt{b}, q^{-n},-q^{-n};q;q^2)\\
=\frac{(-a/b;q)_{2n}}{(-a q;q)_{2n}}
\frac{(a^2q^2,b;q^2)_n}{(1/b,a^2q^2/b^2;q^2)_n} \left(
\frac{q}{b}\right)^n,
\end{multline}
with $a$ replaced with $\sqrt{a}q^{2j}$, $b$ with $a/k$ and $n$ with
$n-j$. The result follows, as above, after some $q$-product
manipulations.
\end{proof}
Inserting the unit pair \eqref{up} in this chain gives the pair
\begin{align}\label{cn333pr}
\alpha_n(a,k,q)&=\frac{1-\sqrt{a}q^n}{1-\sqrt{a}}\frac{\left(\sqrt{a},\frac{a}{k};\sqrt{q}
\right)_n}{\left(\sqrt{q},k\sqrt{\frac{q}{a}};\sqrt{q}
\right)_n}\left( \frac{k}{a}\right)^n,
\\
\beta_n(a,k,q)&=\frac{\left(
-\frac{k}{\sqrt{a}};\sqrt{q}\right)_{2n}}{\left( -\sqrt{a
q};\sqrt{q}\right)_{2n}}\frac{\left(
\frac{a}{k},k;q\right)_{n}}{\left( \frac{k^2
q}{a},q;q\right)_{n}}\left( \frac{k \sqrt{q}}{a}\right)^n, \notag
\end{align}
which is somewhat reminiscent of Bressoud's pair at \eqref{Bpr2}.

Motivated by Warnaar's Theorem 2.6 in \cite{W03}, we also have the
following general WP-Bailey chain.

\begin{theorem}\label{gwct4}
Suppose that $(\alpha_{n}(a,k,q),\,\beta_{n}(a,k,q))$ satisfy
\eqref{WPpair}, that
 $g_n$ is an arbitrary sequence of functions and that $e$ and $c$
 are
arbitrary constants. Then $(\alpha_{n}'(a,k,q)$,
$\beta_{n}'(a,k,q))$ also satisfy \eqref{WPpair}, where
\begin{align}\label{gwp4}
\alpha&_{2n}'(a,k,q) = g_n\alpha_n(e,c,q^2), \hspace{25pt} \alpha_{2n+1}(a,k,q)=0,\\
\beta_{n}'&(a,k,q)= \sum_{j=0}^{n}\beta_{j}(e,c,q^2)\frac{(1-c
q^{4j})\left(\frac{k}{a};q\right)_{n-2j}(k;q)_{n+2j}(eq^2;q^2)_{2j}}{(1-c
)(q;q)_{n-2j}(aq;q)_{n+2j}(cq^2;q^2)_{2j}}\notag\\
\times \sum_{r=0}^{\lfloor n/2 \rfloor-j}&
\frac{(1-eq^{4j+4r})(q^{-(n-2j)},kq^{n+2j};q)_{2r}(e/c,eq^{4j};q^2)_r
g_{r+j}}{(1-e q^{4j})(a q^{n+2j+1},a
q^{1-(n-2j)}/k;q)_{2r}(cq^{4j+2},q^2;q^2)_r}\left(\frac{q^2 c
a^2}{k^2 e} \right)^r.\notag
\end{align}
\end{theorem}
Once again, we omit the proof, since it is very similar to the
proofs of the other theorems above.

Upon setting $g_n=1$, $e=a$ and $c=k^2/a$, and once again using
\eqref{jackson}, we obtain the WP-Bailey chain of Warnaar in Theorem
2.6 of \cite{W03}.
\begin{corollary}[Warnaar, \cite{W03}]\label{cn5}
If $(\alpha_{n}(a,k,q),\,\beta_{n}(a,k,q))$ satisfy \eqref{WPpair},
then so do $(\alpha_{n}'(a,k,q)$, $\beta_{n}'(a,k,q))$, where
{\allowdisplaybreaks
\begin{align}\label{wpn7}
\alpha_{2n}'(a,k,q)&=
\alpha_{n}\left(a,\frac{k^2}{a},q^2\right),\hspace{25pt} \alpha_{2n+1}'(a,k,q)=0\\
\beta_{n}'(a,k,q)&=\frac{\left(\frac{k^2 q}{a};q^2\right)_{n}}
{(aq;q^2)_{n}}\sum_{j=0}^{\lfloor n/2 \rfloor}
\frac{\left(1-\frac{k^2q^{4j}
}{a}\right)}{\left(1-\frac{k^2}{a}\right)}
\frac{\left(\frac{a}{k};q\right)_{n-2j}}{(q;q)_{n-2j}}
\frac{(k;q)_{n+2j}}{\left( \frac{k^2 q
}{a};q\right)_{n+2j}}\notag \\
&\phantom{assddadsASDAsdasdaS} \times \left(
\frac{-k}{a}\right)^{n-2j}
\beta_{j}\left(a,\frac{k^2}{a},q^2\right).\notag
\end{align}
}
\end{corollary}

\section{Transformation Formulae for Basic Hypergeometric Series}

We next describe a number of transformations of basic hypergeometric
series which follow from the new WP-Bailey chains described in
Corollaries \ref{cn11}, \ref{cn2} and \ref{cn333}.

Before proceeding, we remark that the ten or more WP-Bailey chains
that presently  exist potentially give rise to two hundred or more
transformations between basic hypergeometric series, most of which
have probably not been written down (by inserting the ten or more
WP-Bailey pairs and their ``duals" via Corollary \ref{warcor} into
the ten or more existing WP-Bailey chains, ignoring possible
duplication (which is why we write ``potentially")).

Inserting the unit WP-Bailey pair (see \cite{AB02} for example,
where this WP-Bailey pair, and others employed below, may be found),
\begin{align}\label{up}
\alpha_{n}(a,k)&=\frac{(q \sqrt{a}, -q
\sqrt{a},a,a/k;q)_n}{(\sqrt{a},-\sqrt{a},q,kq;q)_n}\left(\frac{k}{a}\right)^n,\\
\beta_n(a,k)&=\begin{cases} 1&n=0, \notag\\
0, &n>1,
\end{cases}
\end{align}
into \eqref{wpn11} leads (perhaps not surprisingly) to \eqref{M1}.
Upon substituting Singh's WP-Bailey pair \cite{S94},
{\allowdisplaybreaks\begin{align}\label{singhpr}
\alpha_{n}(a,k)&=\frac{(q \sqrt{a}, -q
\sqrt{a},a,\rho_1,\rho_2,a^2q/k\rho_1\rho_2;q)_n}
{(\sqrt{a},-\sqrt{a},q,a q/\rho_1,a q/\rho_2,k\rho_1\rho_2/a;q)_n}\left(\frac{k}{a}\right)^n,\\
\beta_n(a,k)&=\frac{(k \rho_1/a, k\rho_2/a, k,
aq/\rho_1\rho_2;q)_n}{(a q/\rho_1, a q/\rho_2, k \rho_1
\rho_2/a,q;q)_n}, \notag
\end{align}
} into \eqref{wpn11} we get the following transformation:
\begin{multline}\label{B127}
_{12} \phi _{11} \left [
\begin{matrix}
a, q \sqrt{a}, -q \sqrt{a}, a \sqrt{\frac{q}{k}}, -a
\sqrt{\frac{q}{k}}, \frac{a}{ \sqrt{k}},-\frac{a}{ \sqrt{k}},\rho_1,
\rho_2,\frac{k q^2}{\rho_1 \rho_2},
k q^n,  q^{-n}\\
\sqrt{a},-\sqrt{a}, \sqrt{k q},-\sqrt{k q},
q\sqrt{k},-q\sqrt{k},\frac{a q}{\rho_1},\frac{a q}{\rho_2},\frac{a
\rho_1 \rho_2}{k q}, \frac{a q^{1-n}}{k}, a q^{n+1}
\end{matrix}
; q,q \right ]\\
=\frac{1-k}{1-k q^{2n}} \frac{\left(a q,\frac{k^2q}{
a^2};q\right)_n}{\left(k,\frac{k}{a};q\right)_n} \, _{7} \phi _{6}
\left [
\begin{matrix}
\frac{a^2}{k q}, q \sqrt{\frac{a^2}{k q}}, -q \sqrt{\frac{a^2}{k
q}}, \frac{a \rho_1}{k q}, \frac{a \rho_2}{k q}, \frac{a q}{ \rho_1
\rho_2}, q^{-n}\\
\sqrt{\frac{a^2}{k q}},-\sqrt{\frac{a^2}{k q}}, \frac{a
q}{\rho_1},\frac{a q}{\rho_2}, \frac{a \rho_1 \rho_2}{k q},
\frac{a^2 q^{-n}}{k^2}
\end{matrix}
; q,q \right ].
\end{multline}
It is easy to recognize this as Bailey's $\,_{12}W_{11} \to \,_7
\phi_6$ transformation (see \cite[III.27]{GR04}). Note that
\eqref{M1} follows immediately  from \eqref{B127}, upon setting
$\rho_1= a q/\rho_2$. Inserting the pair
\begin{align}\label{pr2}
\alpha_{n}(a,k)&=\frac{(q \sqrt{a}, -q \sqrt{a},a,k/aq;q)_n}
{(\sqrt{a},-\sqrt{a},q,a^2q^2/k;q)_n}\frac{(q a^2/k;q)_{2n}}{(k;q)_{2n}}\left(\frac{k}{a}\right)^n,\\
\beta_n(a,k)&=\frac{(k^2/qa^2;q)_n}{(q;q)_n}, \notag
\end{align}
into \eqref{wpn11} leads to
\begin{multline}\label{M2}
_{10} \phi _{9} \left [
\begin{matrix}
a, q \sqrt{a}, -q \sqrt{a}, a \sqrt{\frac{q}{k}}, -a
\sqrt{\frac{q}{k}},  q\sqrt{k q},-q \sqrt{k q},\frac{a}{k q^2},
k q^n,  q^{-n}\\
\sqrt{a},-\sqrt{a}, \sqrt{k q},-\sqrt{k q}, \frac{a}{\sqrt{k
q}},\frac{-a}{\sqrt{k q}},k q^3, \frac{a q^{1-n}}{k}, a q^{n+1}
\end{matrix}
; q,q \right ]\\
=\frac{1-k}{1-k q^{2n}} \frac{\left(a q,\frac{k^2q}{
a^2};q\right)_n}{\left(k,\frac{k}{a};q\right)_n} \, _{4} \phi _{3}
\left [
\begin{matrix}
q \sqrt{\frac{a^2}{k q}}, -q \sqrt{\frac{a^2}{k q}},  \frac{a^2}{
k^2 q^3}, q^{-n}\\
\sqrt{\frac{a^2}{k q}},-\sqrt{\frac{a^2}{k q}}, \frac{a^2
q^{-n}}{k^2}
\end{matrix}
; q,q \right ].
\end{multline}
Similarly, inserting the pair
\begin{align}\label{pr3}
\alpha_{n}(a,k)&=\frac{\left(q \sqrt{a}, -q \sqrt{a},a,a
\sqrt{\frac{q}{k}},-a\sqrt{\frac{q}{k}},\frac{a}{\sqrt{k}},
-\frac{aq}{\sqrt{k}},\frac{k}{a};q\right)_n}
{\left(\sqrt{a},-\sqrt{a},q,\sqrt{kq},-\sqrt{kq},q\sqrt{k},-\sqrt{k},\frac{qa^2}{k};q\right)_n}
\left(\frac{k}{a}\right)^n,\\
\beta_n(a,k)&=\frac{\left(\sqrt{k},\frac{k^2}{a^2};q\right)_n}{(q\sqrt{k},q;q)_n},
\notag
\end{align}
into \eqref{wpn11} leads to
\begin{multline}\label{M3}
_{8} \phi _{7} \left [
\begin{matrix}
a, q \sqrt{a}, -q \sqrt{a},  -a \sqrt{\frac{q}{k}},  -q \sqrt{k
q},\frac{a}{k q},
k q^n,  q^{-n}\\
\sqrt{a},-\sqrt{a}, -\sqrt{k q}, \frac{-a}{\sqrt{k q}},k q^2,
\frac{a q^{1-n}}{k}, a q^{n+1}
\end{matrix}
; q,q \right ]\\
=\frac{1-k}{1-k q^{2n}} \frac{\left(a q,\frac{k^2q}{
a^2};q\right)_n}{\left(k,\frac{k}{a};q\right)_n} \, _{5} \phi _{4}
\left [
\begin{matrix}
q \sqrt{\frac{a^2}{k q}}, -q \sqrt{\frac{a^2}{k q}},
\frac{a}{\sqrt{k q}}, \frac{a^2}{
k^2 q^2}, q^{-n}\\
\sqrt{\frac{a^2}{k q}},-\sqrt{\frac{a^2}{k q}},a \sqrt{\frac{q}{k}},
\frac{a^2 q^{-n}}{k^2}
\end{matrix}
; q,q \right ].
\end{multline}

Substituting the pair
\begin{align}\label{pr4}
\alpha_{n}(a,k)&=\begin{cases}0,&\text{if $n$ is
odd},\\
\frac{\left(q^2 \sqrt{a}, \,-q^2
\sqrt{a},\,a,\,a^2/k^2;\,q^2\right)_{n/2}}
{\left(\sqrt{a},\,-\sqrt{a},\,q^2,\,q^2k^2/a;\,q^2\right)_{n/2}}
\left(\frac{k}{a}\right)^n,&\text{if $n$ is even},
\end{cases}\\
\beta_n(a,k)&=\frac{\left(k,\,k\sqrt{q/a},\,-k\sqrt{q/a},\,a/k;\,q\right)_n}
{(\sqrt{aq},\,-\sqrt{aq},\,q k^2/a,\,q;\,q)_n}
\left(\frac{-k}{a}\right)^n, \notag
\end{align}
into \eqref{wpn11} leads to
\begin{multline}\label{M4}
\text{\negthickspace\tiny{$_{16}W_{15}\left(a;\frac{a}{\sqrt{k}},\frac{-a}{\sqrt{k}},a\sqrt{\frac{q}{k}},-a\sqrt{\frac{q}{k}},\frac{aq}{\sqrt{k}},\frac{-aq}{\sqrt{k}},
aq \sqrt{\frac{q}{k}},-aq \sqrt{\frac{q}{k}},\frac{k^2q^2}{a^2},kq^n,kq^{n+1},q^{-n},q^{1-n};q^2,q^2\right)$}}\\
\negthickspace=\frac{1-k}{1-k q^{2n}} \frac{\left(a q,\frac{k^2q}{
a^2};q\right)_n}{\left(k,\frac{k}{a};q\right)_n} \, _{7} \phi _{6}
\left [
\begin{matrix}
\frac{aq}{\sqrt{k q}}, -\frac{aq}{\sqrt{k q}}, \frac{a^2}{ k q},
\frac{a}{k}\sqrt{\frac{a}{q}},-\frac{a}{k}\sqrt{\frac{a}{q}},
\frac{k q}{a},q^{-n}\\
\frac{a}{\sqrt{k q}},-\frac{a}{\sqrt{k
q}},\sqrt{aq},-\sqrt{aq},\frac{a^3}{k^2q}, \frac{a^2 q^{-n}}{k^2}
\end{matrix}
; q,q \right ],
\end{multline}
a result reminiscent of (3.14) in \cite{AB02}.

We next turn to two WP-Bailey pairs found by Bressoud \cite{B81}.
Inserting the pair
\begin{align}\label{Bpr2}
\alpha_n(a,k)&=\frac{1-\sqrt{a}\,q^n}{1-\sqrt{a}}\,
\frac{\left(\sqrt{a},\frac{a \sqrt{q}}{k};\sqrt{q}\right)_n}
{\left(\sqrt{q},\frac{k}{ \sqrt{a}};\sqrt{q}\right)_n}
\left(\frac{k}{a \sqrt{q}} \right)^n,\\
\beta_n(a,k)&=\frac{\left(k,\frac{a q}{k};q\right)_n}{\left(q,
\frac{k^2}{a};q\right)_n}
\frac{\left(\frac{-k}{\sqrt{a}};\sqrt{q}\right)_{2n}}{\left( -
\sqrt{a q};\sqrt{q}\right)_{2n}}\left(\frac{k}{a \sqrt{q}}
\right)^n, \notag
\end{align}
into \eqref{wpn11}, and then replacing $\sqrt{a}$ with $a$,
$\sqrt{k}$ with $k$ and $\sqrt{q}$ with $q$, gives
{\allowdisplaybreaks \begin{multline}\label{M5} \sum_{j=0}^{n}
\frac{\left(q\sqrt{a}, -q\sqrt{a},a,\frac{k^2q^3}{a^2};q\right)_j}
{\left(\sqrt{a},-\sqrt{a},\frac{a^3}{q^2k^2},q;q\right)_j}
\frac{\left(\frac{a^2}{k},-\frac{a^2}{k},\frac{a^2q}{k},
-\frac{a^2q}{k},k^2q^{2n},q^{-2n};q^2\right)_j q^j}{\left(k q,-k
q,kq^2,-kq^2,\frac{a^2q^{2-2n}}{k^2},a^2q^{2+n};q^2\right)_j}\\
=\frac{1-k^2}{1-k^2 q^{4n}} \frac{\left(a^2 q^2,\frac{k^4q^2}{
a^4};q^2\right)_n}{\left(k^2,\frac{k^2}{a^2};q^2\right)_n}\\ \times
\, _{7} \phi _{6} \left [
\begin{matrix}
\frac{a^2q^2}{k q}, -\frac{a^2q^2}{k q}, \frac{a^4}{ k^2 q^2},
\frac{-a^3}{k^2q^2},\frac{-a^3}{k^2q},
\frac{k^2 q^4}{a^2},q^{-2n}\\
\frac{a^2}{k q},-\frac{a^2}{k q},-a q,-a q^2,\frac{a^6}{k^4q^4},
\frac{a^4 q^{-2n}}{k^4}
\end{matrix}
; q^2,\frac{a^2}{k^2q} \right ],
\end{multline}}
an identity similar to (4.10) in \cite{AB02}. In a similar manner,
inserting Bressoud's WP-Bailey pair
\begin{align}\label{Bpr3}
\alpha_n(a,k)&=\frac{1-a\,q^{2n}}{1-a}\,
\frac{\left(\sqrt{a},\frac{a}{k};\sqrt{q}\right)_n}
{\left(\sqrt{q},k\sqrt{\frac{q}{ a}};\sqrt{q}\right)_n}
\left(\frac{k}{a \sqrt{q}} \right)^n,\\
\beta_n(a,k)&=\frac{\left(k,\frac{a}{k},-k\sqrt{\frac{q}{a}},-\frac{k
q}{\sqrt{a}};q\right)_n}{\left(q, \frac{q k^2}{a},-\sqrt{a},-\sqrt{a
q};q\right)_n} \left(\frac{k}{a \sqrt{q}} \right)^n,\notag
\end{align}
into \eqref{wpn11}, and likewise replacing $\sqrt{a}$ with $a$,
$\sqrt{k}$ with $k$ and $\sqrt{q}$ with $q$, gives
{\allowdisplaybreaks \begin{multline}\label{M6} \sum_{j=0}^{n}
\frac{\left(a,\frac{k^2q^2}{a^2};q\right)_j}
{\left(\frac{a^3}{qk^2},q;q\right)_j} \frac{\left(a q^2, -a
q^2,\frac{a^2}{k},-\frac{a^2}{k},\frac{a^2q}{k},
-\frac{a^2q}{k},k^2q^{2n},q^{-2n};q^2\right)_j q^j}{\left(a,-a,k
q,-k
q,kq^2,-kq^2,\frac{a^2q^{2-2n}}{k^2},a^2q^{2+n};q^2\right)_j}\\
=\frac{1-k^2}{1-k^2 q^{4n}} \frac{\left(a^2 q^2,\frac{k^4q^2}{
a^4};q^2\right)_n}{\left(k^2,\frac{k^2}{a^2};q^2\right)_n}\\ \times
\, _{7} \phi _{6} \left [
\begin{matrix}
\frac{a^2q}{k}, -\frac{a^2q}{k}, \frac{a^4}{ k^2 q^2},
\frac{-a^3}{k^2q},\frac{-a^3}{k^2},
\frac{k^2 q^2}{a^2},q^{-2n}\\
\frac{a^2}{k q},-\frac{a^2}{k q},-a,-a q,\frac{a^6}{k^4q^2},
\frac{a^4 q^{-2n}}{k^4}
\end{matrix}
; q^2,\frac{a^2}{k^2q} \right ].
\end{multline}
}

Finally, we apply the corollary to two WP-Bailey pairs found by the
present authors in \cite{MZ07b}:
\begin{align}\label{mz01}
\alpha_n^{(1)}(a,k)&=\frac{(q a^2/k^2;q)_n}{(q,q)_n}\left(
\frac{k}{a}\right)^n,\\
\beta_n^{(1)}(a,k)&=\frac{(q
a/k,k;q)_n}{(k^2/a,q,q)_n}\frac{(k^2/a;q)_{2n}}{(a q,q)_{2n}}.\notag
\end{align}

\begin{align}\label{mz02}
\alpha_n^{(2)}(a,k)&=\frac{(a, \,q \sqrt{a},\, -q \sqrt{a},
\,k/a,\,a \sqrt{q/k}, \, -a \sqrt{q/k};q)_n}
{(\sqrt{a},\,-\sqrt{a},\, q
a^2/k,\,\sqrt{qk},\, -\sqrt{qk},\,q;q)_n}\,(-1)^n,\\
\beta_n^{(2)}(a,k)&=\begin{cases}
\displaystyle{\frac{(k,k^2/a^2;q^2)_{n/2}}{(q^2, q^2
a^2/k;q^2)_{n/2}}}, & n \text{ even},\\
0,& n \text{ odd}.
\end{cases}\notag
\end{align}

Note that the second pair may also be derived by applying Corollary
\ref{warcor} to the pair at \eqref{pr4}. These WP-Bailey pairs
inserted in  \eqref{wpn11} lead, respectively to the transformations
{\allowdisplaybreaks
\begin{multline}\label{M7} _{7} \phi _{6} \left
[
\begin{matrix}
\frac{a}{\sqrt{k}}, \frac{-a}{\sqrt{k}},a \sqrt{\frac{q}{k}}, -a
\sqrt{\frac{q}{k}},  \frac{k^2q^3}{a^2},
k q^n,  q^{-n}\\
 \sqrt{k q}, -\sqrt{k q},  q\sqrt{k}, -q\sqrt{k},
\frac{a q^{1-n}}{k}, a q^{n+1}
\end{matrix}
; q,q\right ]\\
=\frac{1-k}{1-k q^{2n}} \frac{\left(a q,\frac{k^2q}{
a^2};q\right)_n}{\left(k,\frac{k}{a};q\right)_n}\\
\times \, _{9} \phi _{8} \left [
\begin{matrix}
\frac{a q}{\sqrt{k q}}, - \frac{a q}{\sqrt{k q}}, \frac{a^2}{k q},
\frac{a^{3/2}}{ k q}, \frac{-a^{3/2}}{ k q},\frac{a^{3/2}}{ k
\sqrt{q}}, \frac{-a^{3/2}}{
k \sqrt{q}},\frac{k q^2}{a},q^{-n}\\
\frac{a}{\sqrt{k q}}, - \frac{a}{\sqrt{k q}},\sqrt{a q},-\sqrt{a q},
q\sqrt{a},-q\sqrt{a}, \frac{a^3}{k^2q^2},\frac{a^2 q^{-n}}{k^2}
\end{matrix}
; q,q \right ],
\end{multline}}
and {\allowdisplaybreaks
\begin{multline}\label{M8} _{8} \phi _{7}
\left [
\begin{matrix}
a,q\sqrt{a},-q\sqrt{a},a \sqrt{\frac{q}{k}}, -a \sqrt{\frac{q}{k}},
\frac{a}{k q},
k q^n,  q^{-n}\\
\sqrt{a},-\sqrt{a}, \sqrt{k q}, -\sqrt{k q}, q^2 k,\frac{a
q^{1-n}}{k}, a q^{n+1}
\end{matrix}
; q,-\frac{k q^2}{a}  \right ]\\
=\frac{1-k}{1-k q^{2n}} \frac{\left(a q,\frac{k^2q}{
a^2};q\right)_n}{\left(k,\frac{k}{a};q\right)_n}
\,_{6} \phi _{5} \left [
\begin{matrix}
\frac{a q^2}{\sqrt{k q}}, - \frac{a q^2}{\sqrt{k q}}, \frac{a^2}{k
q}, \frac{a^2}{k^2q^2},q^{-n},q^{1-n}\\
\frac{a}{\sqrt{k q}}, - \frac{a}{\sqrt{k q}}, q^3k,\frac{a^2
q^{-n}}{k^2},\frac{a^2 q^{1-n}}{k^2}
\end{matrix}
; q^2,q^2 \right ].
\end{multline}
}

Likewise, inserting the WP-Bailey pairs at \eqref{singhpr},
\eqref{pr2}, \eqref{pr3}, \eqref{pr4}, \eqref{Bpr2}, \eqref{Bpr3},
\eqref{mz01} and \eqref{mz02}, respectively, into the WP-Bailey
chain at \eqref{wpn5} leads, respectively, to the following
transformations:
\begin{multline}\label{singh2}
_{10} W _{9} \left(\frac{a k}{q};-a,-a q,\frac{k \rho_1}{a q},
\frac{k \rho_2}{a q},\frac{a^2 q^2}{\rho_1 \rho_2},k^2 q^{2n},
q^{-2n};q^2, \frac{q^2 a}{k}
 \right)\\
=\frac{\left(k/a,-q;q\right)_n}{\left(-k,a q;q\right)_n} \frac{(a k
q;q^2)_n}{(k q/a;q^2)_n}\\
\times \sum_{j=0}^{n} \frac{(k q^n,q^{-n};q)_j} {\left( \frac{a
q^{1-n}}{k},a q^{1+n};q\right)_j}
\frac{(a^2,aq^2,\rho_1,\rho_2,\frac{a^3q^3}{k \rho_1 \rho_2};q^2)_j}
{\left(a,\frac{a^2 q^2}{\rho_1},\frac{a^2 q^2}{\rho_2},\frac{k
\rho_1 \rho_2}{a q},q^2;q^2\right)_j}\,q^j;
\end{multline}
 {\allowdisplaybreaks\begin{multline}\label{pr22} _{7} \phi _{6}
\left[
\begin{matrix} q^2\sqrt{\frac{a k}{q}}, -q^2\sqrt{\frac{a k}{q}},-a,
-a q, \frac{k^2}{a^2 q^4},k^2
q^{2n},q^{-2n}\\
\sqrt{\frac{a k}{q}}, -\sqrt{\frac{a k}{q}},-k, -k q, a k
q^{2n+1},\frac{a q^{1-2n}}{k}
\end{matrix};q^2, \frac{q^2 a}{k} \right]\\
=\frac{\left(k/a,-q;q\right)_n}{\left(-k,a q;q\right)_n} \frac{(a k
q;q^2)_n}{(k q/a;q^2)_n}\\
\times \sum_{j=0}^{n} \frac{(k q^n,q^{-n};q)_j} {\left( \frac{a
q^{1-n}}{k},a q^{1+n};q\right)_j} \frac{(a^2,aq^2,\frac{k}{a
q^3};q^2)_j} {\left(a,\frac{a^3 q^5}{k},q^2;q^2\right)_j}
\frac{\left(\frac{q^3a^3}{k},\frac{q^5a^3}{k};q^4\right)_j}
{\left(\frac{a k}{q},a k q;q^4\right)_j}\,q^j;
\end{multline}}
\allowdisplaybreaks{
\begin{multline}\label{pr32}
_{6} \phi _{5} \left[ \begin{matrix} -q^2\sqrt{\frac{a k}{q}},-a, -a
q, \frac{k^2}{a^2 q^2},k^2
q^{2n},q^{-2n}\\
 -\sqrt{\frac{a k}{q}},-k, -k q, a k
q^{2n+1},\frac{a q^{1-2n}}{k}
\end{matrix};q^2, \frac{q^2 a}{k} \right]\\
=\frac{\left(k/a,-q;q\right)_n}{\left(-k,a q;q\right)_n} \frac{(a k
q;q^2)_n}{(k q/a;q^2)_n}\\
\times \sum_{j=0}^{n} \frac{(k q^n,q^{-n};q)_j} {\left( \frac{a
q^{1-n}}{k},a q^{1+n};q\right)_j} \frac{\left(a^2,aq^2,\frac{k}{a
q},\sqrt{\frac{a^3q}{k}},-\sqrt{\frac{a^3q^5}{k}};q^2\right)_j}
{\left(a,\frac{a^3 q^3}{k},\sqrt{a k q^3},-\sqrt{\frac{a
k}{q}},q^2;q^2\right)_j} \frac{\left(\frac{q^3a^3}{k};q^4\right)_j}
{\left(a k q;q^4\right)_j}\,q^j;
\end{multline}
}
\begin{multline}\label{pr42}
_8W_7\left(\frac{a k}{q};-a,k,\frac{a
q}{k},k^2q^{2n},q^{-2n};q^2,-q\right)=\frac{\left(k/a,-q;q\right)_n}{\left(-k,a
q;q\right)_n} \frac{(a k
q;q^2)_n}{(k q/a;q^2)_n}\\
\times _{10}W_9 \left(a;-a,\frac{a q}{k},-\frac{a q}{k}, k q^{n}, k
q^{1+n}, q^{1-n},q^{-n};q^2,q^2 \right);
\end{multline}
\begin{multline}\label{Bpr22}
_8W_7\left(\frac{a k}{q};-a,-\frac{k}{q},\frac{a
q^3}{k},k^2q^{2n},q^{-2n};q^2,1\right)=\frac{\left(k/a,-q;q\right)_n}{\left(-k,a
q;q\right)_n} \frac{(a k
q;q^2)_n}{(k q/a;q^2)_n}\\
\times _{8}W_7 \left(a;i\sqrt{a},-i\sqrt{a},\frac{a q^2}{k}, k
q^{n},q^{-n};q,1 \right);
\end{multline}
\begin{multline}\label{Bpr23}
_6W_5\left(\frac{a k}{q};\frac{a
q}{k},k^2q^{2n},q^{-2n};q^2,1\right)=\frac{\left(k/a,-q;q\right)_n}{\left(-k,a
q;q\right)_n} \frac{(a k
q;q^2)_n}{(k q/a;q^2)_n}\\
\times _{6}W_5 \left(a;\frac{a q}{k}, k q^{n},q^{-n};q,1 \right);
\end{multline}
\begin{multline}\label{MZ012}
_{10}W_9\left(\frac{a k}{q};-a,k,\frac{k}{q},-\frac{k}{q},\frac{a
q^3}{k},k^2q^{2n},q^{-2n};q^2,\frac{q^2a}{k}\right)=\\\frac{\left(k/a,-q;q\right)_n}{\left(-k,a
q;q\right)_n} \frac{(a k q;q^2)_n}{(k q/a;q^2)_n} \sum_{j=0}^{n}
\frac{(1+a)\,\,(aq^2/k,-aq^2/k,kq^n,q^{-n};q)_j}{(1+a
q^{2j})(aq^{1-n}/k,aq^{1+n},-q,q;q)_j}\,q^j;
\end{multline}
\begin{multline}\label{MZ022}
_{12}W_{11}\left(\text{\small{$\frac{a k}{q};-a,-a
q,-aq^2,-aq^3,\frac{k^2}{a^2q^2},k^2q^{2n},k^2q^{2+2n},
q^{2-2n},q^{-2n};q^4,\frac{q^4a^2}{k^2}$}}\right)\\=\frac{\left(k/a,-q;q\right)_n}{\left(-k,a
q;q\right)_n} \frac{(a k q;q^2)_n}{(k q/a;q^2)_n}\phantom{asdadasdasdasd}\phantom{asdadasdasdasd}\\
\times \sum_{j=0}^{n} \frac{(kq^n,q^{-n};q)_j}
{\left(\frac{aq^{1-n}}{k},aq^{1+n};q\right)_j}\,\frac{\left(a^2,aq^2,\frac{k}{a
q};q^2\right)_j} {\left(a,\frac{a^3q^3}{k},q^2;q^2\right)_j}\,
\frac{\left(\frac{a^3q^3}{k};q^4\right)_j} {\left(a k
q;q^4\right)_j} \left(-\frac{q^2a}{k} \right)^j.
\end{multline}

In a similar manner, inserting the WP-Bailey pairs at
\eqref{singhpr}, \eqref{pr2}, \eqref{pr3}, \eqref{pr4},
\eqref{Bpr2}, \eqref{Bpr3}, \eqref{mz01} and \eqref{mz02},
respectively, into the WP-Bailey chain at \eqref{wpn6aa} leads,
respectively, to the following transformations: 
\begin{multline}\label{BT3.742}
\, _{10} \phi _{9} \left [
\begin{matrix}a,q \sqrt{a}, -q\sqrt{a},\rho_1,\rho_2,\frac{a^3q}{k\rho_1\rho_2},\sqrt{k}q^n,-\sqrt{k}q^n,q^{-n},-q^{-n}\\
\sqrt{a},-\sqrt{a},\frac{a q}{\rho_1},\frac{a
q}{\rho_2},\frac{k\rho_1\rho_2}{a^2},\frac{aq^{1-n}}{\sqrt{k}},-\frac{aq^{1-n}}{\sqrt{k}},aq^{n+1},-aq^{n+1}
\end{matrix}
; q,q^2 \right ]\\
=\frac{\left(a^2q^2,\frac{a^2}{k};q^2\right)_n\left(\frac{-k}{a};q\right)_{2n}}
{\left(\frac{k}{a^2},\frac{k^2q^2}{a^2};q^2\right)_n\left(-qa;q\right)_{2n}}\left(\frac{kq}{a^2}\right)^n
\phantom{asdadasdasdasd}\\
\times \sum_{j=0}^{n}
\frac{\left(1-\frac{k^2}{a^2}q^{4j}\right)\left(kq^{2n},q^{-2n};q^2\right)_{j}\left(\frac{k
\rho_1}{a^2}, \frac{k\rho_2}{a^2},
\frac{k}{a},\frac{aq}{\rho_1\rho_2};q\right)_j}
{\left(1-\frac{k^2}{a^2}\right)\left(\frac{k^2q^{2n+2}}{a^2},\frac{kq^{2-2n}}{a^2};q^2\right)_{j}\left(\frac{a
q}{\rho_1}, \frac{a q}{\rho_2}, \frac{k
\rho_1\rho_2}{a^2},q;q\right)_j}\, q^{j};
\end{multline}
\begin{multline}\label{BT3.744}
\sum_{j=0}^{n} \frac{\left(a,q \sqrt{a}, -q
\sqrt{a},\frac{k}{a^2q};q\right)_j\left(q^{-2n},kq^{2n};q^2\right)_{j}\left(\frac{q
a^3}{k};q\right)_{2j}}
{\left(\sqrt{a},-\sqrt{a},q,\frac{a^3q^2}{k};q\right)_j\left(\frac{a^2q^{2-2n}}{k},a^2q^{2n+2};q^2\right)_{j}\left(\frac{k}{a};q\right)_{2j}}
\,q^{2j}\\
=
\frac{\left(a^2q^2,\frac{a^2}{k};q^2\right)_n\left(\frac{-k}{a};q\right)_{2n}}
{\left(\frac{k}{a^2},\frac{k^2q^2}{a^2};q^2\right)_n\left(-qa;q\right)_{2n}}
\left(\frac{kq}{a^2}\right)^n\phantom{asdadasdasdasd}\\
\times \sum_{j=0}^{n}
\frac{\left(1-\frac{k^2}{a^2}q^{4j}\right)\left(kq^{2n},q^{-2n};q^2\right)_{j}\left(\frac{k^2}{qa^4};q\right)_j}
{\left(1-\frac{k^2}{a^2}\right)\left(\frac{k^2q^{2n+2}}{a^2},\frac{kq^{2-2n}}{a^2};q^2\right)_{j}\left(q;q\right)_j}
q^{j};
\end{multline}
\begin{multline}\label{BT3.746}
\text{\scriptsize{$_{12}\phi
_{11}\left[\begin{matrix}a,q\sqrt{a},-q\sqrt{a},a\sqrt{\frac{qa}{k}},-a\sqrt{\frac{qa}{k}},a\sqrt{\frac{a}{k}},-aq\sqrt{\frac{a}{k}},\frac{k}{a^2},q^{-n},-q^{-n},\sqrt{k}q^n,-\sqrt{k}q^n
\\
\sqrt{a},-\sqrt{a},\sqrt{\frac{kq}{a}},-\sqrt{\frac{kq}{a}},q\sqrt{\frac{k}{a}},-\sqrt{\frac{k}{a}},\frac{qa^3}{k},\frac{aq^{1-n}}{\sqrt{k}},-\frac{aq^{1-n}}{\sqrt{k}},aq^{n+1},-aq^{n+1}
\end{matrix};q,q^2\right]$}}\\
=\frac{\left(a^2q^2,\frac{a^2}{k};q^2\right)_n\left(\frac{-k}{a};q\right)_{2n}}
{\left(\frac{k}{a^2},\frac{k^2q^2}{a^2};q^2\right)_n\left(-qa;q\right)_{2n}}\left(\frac{kq}{a^2}\right)^n
\phantom{asdadasdasdasd}\\
\times \sum_{j=0}^{n}
\frac{\left(1-\frac{k^2}{a^2}q^{4j}\right)\left(kq^{2n},q^{-2n};q^2\right)_{j}}
{\left(1-\frac{k^2}{a^2}\right)\left(\frac{k^2q^{2n+2}}{a^2},\frac{kq^{2-2n}}{a^2};q^2\right)_{j}}
\frac{\left(\sqrt{\frac{k}{a}},\frac{k^2}{a^4};q\right)_j}{\left(q\sqrt{\frac{k}{a}},q;q\right)_j}\,q^j;
\end{multline}
{\allowdisplaybreaks
\begin{multline}\label{BT3.748}
\sum_{j=0}^{n}
\frac{\left(q^{-4n},kq^{4n};q^2\right)_{2j}\left(a,\,q^2 \sqrt{a},
\,-q^2\sqrt{a},\,\frac{a^4}{k^2};\,q^2\right)_{j}}
{\left(\frac{a^2q^{2-4n}}{k},a^2q^{4n+2};q^2\right)_{2j}\left(q^2,\,\sqrt{a},\,-\sqrt{a},\,\frac{q^2k^2}{a^3};\,q^2\right)_{j}}
\,q^{4j}\\
=
\frac{\left(a^2q^2,\frac{a^2}{k};q^2\right)_{2n}\left(\frac{-k}{a};q\right)_{4n}}
{\left(\frac{k}{a^2},\frac{k^2q^2}{a^2};q^2\right)_{2n}\left(-qa;q\right)_{4n}}
\left(\frac{kq}{a^2}\right)^{2n} \phantom{asdadasdasdasd}\\
\times \sum_{j=0}^{2n}
\frac{\left(1-\frac{k^2}{a^2}q^{4j}\right)\left(kq^{4n},q^{-4n};q^2\right)_{j}\left(\frac{k}{a},\,\frac{k\sqrt{a
q}}{a^2},\,-\frac{k\sqrt{a q}}{a^2},\,\frac{a^2}{k};\,q\right)_j}
{\left(1-\frac{k^2}{a^2}\right)\left(\frac{k^2q^{4n+2}}{a^2},\frac{kq^{2-4n}}{a^2};q^2\right)_{j}\left(\sqrt{aq},\,-\sqrt{aq},\,\frac{q
k^2}{a^3},\,q;\,q\right)_j} \left(\frac{-kq}{a^2}\right)^j;
\end{multline}
}
\begin{multline}\label{BT3.7410}
\sum_{j=0}^{n}
\frac{\left(1-\sqrt{a}\,q^{2n}\right)\left(q^{-4n},kq^{4n};q^4\right)_{j}\left(\sqrt{a},\frac{a^2
q}{k};q\right)_j}
{\left(1-\sqrt{a}\right)\left(\frac{a^2q^{4-4n}}{k},a^2q^{4n+4};q^4\right)_{j}\left(q,\frac{k\sqrt{a}}{
a^2};q\right)_j}\, q^{3j}\\
=
\frac{\left(a^2q^4,\frac{a^2}{k};q^4\right)_n\left(\frac{-k}{a};q^2\right)_{2n}}
{\left(\frac{k}{a^2},\frac{k^2q^4}{a^2};q^4\right)_n\left(-q^2a;q^2\right)_{2n}}
\left(\frac{kq^2}{a^2}\right)^n\phantom{asdadasdasdasd}
\\
\times \sum_{j=0}^{n}
\frac{\left(1-\frac{k^2}{a^2}q^{8j}\right)\left(kq^{4n},q^{-4n};q^4\right)_{j}\left(\frac{k}{a},\frac{a^2
q^2}{k};q^2\right)_j\left(\frac{-k\sqrt{a}}{a^2};q\right)_{2j}}
{\left(1-\frac{k^2}{a^2}\right)\left(\frac{k^2q^{4n+4}}{a^2},\frac{kq^{4-4n}}{a^2};q^4\right)_{j}\left(q^2,\frac{k^2}{a^3};q^2\right)_j\left(
-q\sqrt{a};q\right)_{2j}} \left(\frac{kq}{a^2}\right)^{j};
\end{multline}
\begin{multline}\label{BT3.7412}
\sum_{j=0}^{n}
\frac{\left(1-aq^{4j}\right)\left(q^{-4n},kq^{4n};q^4\right)_{j}\left(\sqrt{a},\frac{a^2}{k};q\right)_j}
{\left(1-a\right)\left(\frac{a^2q^{4-4n}}{k},a^2q^{4n+4};q^4\right)_{j}\left(q,\frac{kq\sqrt{a}}{a^2};q\right)_j}
\,q^{3j}\\
 =
\frac{\left(a^2q^4,\frac{a^2}{k};q^4\right)_n\left(\frac{-k}{a};q^2\right)_{2n}}
{\left(\frac{k}{a^2},\frac{k^2q^4}{a^2};q^4\right)_n\left(-q^2a;q^2\right)_{2n}}
\left(\frac{kq^2}{a^2}\right)^n\phantom{asdadasdasdasd}
\\
\times  \sum_{j=0}^{n}
\frac{\left(1-\frac{k^2}{a^2}q^{8j}\right)\left(kq^{4n},q^{-4n};q^4\right)_{j}\left(\frac{k}{a},\frac{a^2}{k},-\frac{kq\sqrt{a}}{a^2},-\frac{kq^2\sqrt{a}}{a^2};q^2\right)_j}
{\left(1-\frac{k^2}{a^2}\right)\left(\frac{k^2q^{4n+4}}{a^2},\frac{kq^{4-4n}}{a^2};q^4\right)_{j}\left(q^2,
\frac{q^2 k^2}{a^3},-\sqrt{a},-q\sqrt{a};q^2\right)_j} \left(
\frac{kq}{a^2 }\right)^{j};
\end{multline}
\begin{multline}\label{BT3.7414}
 _{5} \phi _{4} \left [
\begin{matrix}\frac{qa^4}{k^2},q^{-n},-q^{-n},\sqrt{k}q^n,-\sqrt{k}q^n
\\
\frac{aq^{1-n}}{\sqrt{k}},-\frac{aq^{1-n}}{\sqrt{k}},aq^{n+1},-aq^{n+1}
\end{matrix}
; q,q^2 \right ]=
\frac{\left(a^2q^2,\frac{a^2}{k};q^2\right)_{n}\left(\frac{-k}{a};q\right)_{2n}}
{\left(\frac{k}{a^2},\frac{k^2q^2}{a^2};q^2\right)_{n}\left(-qa;q\right)_{2n}}
\\
\times \left(\frac{kq}{a^2}\right)^{n}\sum_{j=0}^{n}
\frac{\left(1-\frac{k^2}{a^2}q^{4j}\right)\left(kq^{2n},q^{-2n};q^2\right)_{j}\left(\frac{qa^2}{k},\frac{k}{a};q\right)_j\left(\frac{k^2}{a^3};q\right)_{2j}}
{\left(1-\frac{k^2}{a^2}\right)\left(\frac{k^2q^{2n+2}}{a^2},\frac{kq^{2-2n}}{a^2};q^2\right)_{j}\left(\frac{k^2}{a^3},q;q\right)_j\left(a
q;q\right)_{2j}} \,q^{j};
\end{multline}
{\allowdisplaybreaks
\begin{multline}\label{BT3.7415}
\text{\footnotesize{$_{10}\phi
_{9}\left[\begin{matrix}a,q\sqrt{a},-q
\sqrt{a},\frac{k}{a^2},a\sqrt{\frac{aq}{k}},-a\sqrt{\frac{aq}{k}},q^{-2n},-q^{-2n},\sqrt{k}q^{2n},-\sqrt{k}q^{2n}
\\
\sqrt{a},-\sqrt{a},\frac{qa^3}{k},\sqrt{\frac{qk}{a}},-\sqrt{\frac{qk}{a}},q,\frac{aq^{1-2n}}{\sqrt{k}},-\frac{aq^{1-2n}}{\sqrt{k}},aq^{2n+1},-aq^{2n+1}
\end{matrix}
; q,\frac{-q^2a^2}{k}\right]$}}\\
=\frac{\left(a^2q^2,\frac{a^2}{k};q^2\right)_{2n}\left(\frac{-k}{a};q\right)_{4n}}
{\left(\frac{k}{a^2},\frac{k^2q^2}{a^2};q^2\right)_{2n}\left(-qa;q\right)_{4n}}\left(\frac{kq}{a^2}\right)^{2n}
\phantom{asdadasdasdasd}\\
\times \sum_{j=0}^{n}
\frac{\left(1-\frac{k^2}{a^2}q^{8j}\right)\left(kq^{4n},q^{-4n};q^2\right)_{2j}\left(\frac{k}{a},\frac{k^2}{a^4};q^2\right)_{j}}
{\left(1-\frac{k^2}{a^2}\right)\left(\frac{k^2q^{4n+2}}{a^2},\frac{kq^{2-4n}}{a^2};q^2\right)_{2j}\left(q^2,
\frac{q^2a^3}{k};q^2\right)_{j}}\, q^{2j}.
\end{multline}
}

\section{Bailey-Type Chains}
Just as Liu and Ma did in \cite{LM08}, we can let $k \to 0$ to get
regular Bailey chains, but unfortunately none are new.

\begin{corollary}\label{cn4} If $(\alpha_{n}(a,q),\,\beta_{n}(a,q))$ are a Bailey pair
w.r.t. $a$, then so is $(\alpha_{n}'(a,q)$, $\beta_{n}'(a,q))$,
where {\allowdisplaybreaks
\begin{align}\label{pn1}
\alpha_{n}'(a,q)&=\frac{(1+a)q^n}{1+a q^{2n}}\alpha_{n}(a^2,q^2),\\
\beta_{n}'(a,q)&=\sum_{j=0}^{n} \frac{(-a;q)_{2j}}{(q^2;q^2)_{n-j}}
 q^j\beta_{j}(a^2,q^2).\notag
\end{align}
}
\end{corollary}
\begin{proof}
let $k \to 0$ in Corollary \ref{cn2} and rearrange.
\end{proof}
Remark: This is the Bailey chain (D4) of Bressoud, Ismail and
Stanton in \cite{BIS00}.

\begin{corollary}\label{cn44} If $(\alpha_{n}(a,q),\,\beta_{n}(a,q))$ are a Bailey pair
w.r.t. $a$, then so is $(\alpha_{n}'(a,q)$, $\beta_{n}'(a,q))$,
where {\allowdisplaybreaks
\begin{align}\label{pn2}
\alpha_{n}'(a^2,q^2)&=\alpha_{n}(a,q),\\
\beta_{n}'(a^2,q^2)&=\frac{1}{(-aq;q)_{2n}}\sum_{j=0}^{n}
\frac{(-1)^{n-j}q^{(n-j)^2}}{(q^2;q^2)_{n-j}}
 \beta_{j}(a,q).\notag
\end{align}
}
\end{corollary}
\begin{proof}
let $k \to 0$ in Corollary \ref{cn333} and rearrange.
\end{proof}
We  note that is chain is also not new, being  essentially the
Bailey chain (D1) in \cite{BIS00}.

The corollaries above and other results in the literature (for
example the first WP-Bailey chain of Andrews in \cite{A01}) show
that some Bailey chains may be ``lifted" to  WP-Bailey chains.

Is it possible to lift \emph{all} Bailey chains to WP-Bailey chains?

 \allowdisplaybreaks{

}
\end{document}